\documentclass[12pt]{amsart}

\usepackage{url}
\usepackage{graphicx}
\usepackage{amssymb, amscd}
\usepackage[dvipsnames]{xcolor}
\usepackage[lmargin=1.25in,rmargin=1.25in,tmargin=1.25in,bmargin=1in,dvips]{geometry}
\usepackage[dvips, color]{xy}
\xyoption{all}
\usepackage{bm}
\usepackage{pinlabel}
\usepackage[linktocpage]{hyperref}

\DeclareMathOperator{\rank}{rank}
\let\hom\relax
\DeclareMathOperator{\hom}{Hom}
\DeclareMathOperator{\coker}{coker}

\def\co{\colon\thinspace}

\newcommand{\PD}{\operatorname{PD}}

\newcommand{\Z}{\mathbb Z}
\newcommand{\Q}{\mathbb Q}

\def\spinc{\ifmmode{\textrm{Spin}^c}\else{$\textrm{Spin}^c$}\fi}
\newcommand{\spincs}{\mathfrak s}
\newcommand{\spinct}{\mathfrak t}

\newcommand{\spincu}{\mathfrak u}

\renewcommand{\phi}{\varphi}

\newtheorem{theorem}{Theorem}[section]

\newtheorem{lemma}[theorem]{Lemma}

\newtheorem{proposition}[theorem]{Proposition}
\newtheorem{corollary}[theorem]{Corollary}
\theoremstyle{definition}
\newtheorem{definition}[theorem]{Definition}

\newtheorem{example}[theorem]{Example}

\numberwithin{equation}{section}

\def\K{\mathbb{K}}
\def\L{\mathbb{L}}

\makeatletter
\let\@wraptoccontribs\wraptoccontribs
\makeatother

\DeclareMathOperator{\Spin}{Spin}
\def\s{\mathfrak{s}}

\def\CF {\operatorname{CF}}

\def\HF {\operatorname{HF}}

\def\HFp {\operatorname{HF}^+}
\def\HFm {\operatorname{HF}^-}
\def\HFi {\operatorname{HF}^\infty}

\def\CFp {\operatorname{CF}^+}
\def\CFm {\operatorname{CF}^-}
\def\CFi {\operatorname{CF}^\infty}

\def\dtop {d_{\operatorname{top}}}
\def\dbot {d_{\operatorname{bot}}}

\DeclareMathOperator{\Tors}{Tors}
\DeclareMathOperator{\gr}{gr}
\DeclareMathOperator{\nbd}{nbd}
\DeclareMathOperator{\im}{im}

\DeclareMathOperator{\st}{st}
\def\HFcirc {\operatorname{HF}^\circ}

\def\DD {\mathcal{D}}
\def\II {\mathcal{I}}
\def\KK {\mathcal{K}}
\def\QQ {\mathcal{Q}}
\def\KQ {\mathcal{KQ}}
\def\QK {\mathcal{QK}}

\newcommand{\abs}[1] {\left\lvert #1 \right\rvert}

\newcommand{\gen}[1] {\langle #1 \rangle}
\def\conn{\mathbin{\#}}

\def\minus{\smallsetminus}

\newcommand{\arxiv}[1]{\href{http://arxiv.org/abs/#1}{arXiv:#1}}

\title{Generalized Heegaard Floer correction terms}

\author[Adam Simon Levine]{Adam Simon Levine}
\address{Department of Mathematics \newline\indent Princeton University \newline\indent Princeton, NJ 08540}
\email{asl2@math.princeton.edu}
\author[Daniel Ruberman]{Daniel Ruberman}
\address{Department of Mathematics, MS 050\newline\indent Brandeis
University \newline\indent Waltham, MA 02454}
\email{ruberman@brandeis.edu}
\thanks{The first author was partially supported by an NSF Postdoctoral Fellowship.  The second author was partially supported by NSF Grant 1105234.
\\
Math.~Subj.~Class.~2010: 57M27, 57R58, 57Q60.}

\begin{document}
\begin{abstract}
We make use of the action of $H_1(Y)$ in Heegaard Floer homology to generalize the Ozsv\'ath-Szab\'o correction terms for $3$-manifolds with standard $\HFi$, introducing {\em intermediate $d$-invariants}.   We establish the basic properties of these invariants: conjugation invariance, behavior under orientation reversal, additivity, and \spinc-rational homology cobordism invariance.
\end{abstract}
\maketitle

\section{Introduction}

In this paper, we study a generalized version of the $d$-invariants, or correction terms, defined by Ozsv\'ath and Szab\'o in~\cite{oz:boundary} via gradings in Heegaard Floer theory.  These invariants --- a rational number $d(Y, \spincs)$ associated to each \spinc\ structure $\spincs$ on a rational homology $3$-sphere $Y$ --- have proven to be very useful in low-dimensional topology.  The numbers $d(Y,\spincs)$ are (\spinc) rational homology cobordism invariants; they are the Heegaard Floer analogues of the $h$-invariants introduced by Fr{\o}yshov~\cite{froyshov:equivariant,froyshov:boundary} in the setting of Yang-Mills or Seiberg-Witten gauge theory.  By passing to branched covers or surgery on a knot $K$, one derives knot invariants that can be used to elucidate the difference between smooth and topological knot concordance or bound the smooth $4$-genus of a knot; see for instance~\cite{manolescu-owens:delta, jabuka-naik:dcover, hedden-livingston-ruberman:alexander, hedden-kim-livingston:2-torsion, peters:d-invariant, levine:infinite, owens-strle:qhs4b}.

As explained in~\cite{oz:boundary}, certain $d$-invariants can be defined for a $3$-manifold $Y$ that is not a rational homology sphere, as long as the \spinc\ structure has torsion first Chern class, and the Floer homology of $Y$ is `standard' in a sense that we will explain below in Section~\ref{S:correction}.  The simplest such invariant, denoted in the current paper by $\dbot(Y,\spincs)$, was introduced in~\cite[\S 9]{oz:boundary} via an action of $H_1(Y)/\mathrm{torsion}$. By computing this invariant for a particular \spinc\ structure on a non-trivial oriented circle bundle over an oriented surface, Ozsv\'ath and Szab\'o gave a new proof of the Thom conjecture~\cite{kronheimer-mrowka:thom}, along the lines introduced in~\cite{strle:intersections}.   A recent paper by the authors and S.~Strle~\cite{levine-ruberman-strle:surfaces} makes use of $\dbot(Y,\spincs)$ and a `dual' invariant $\dtop(Y,\spincs)$ to study embeddings of non-orientable surfaces in $4$-manifolds.  A key component of that paper is the calculation of these correction terms for a torsion \spinc\ structure on a non-orientable circle bundle over a non-orientable surface.

Experts in the field have understood that there should be a more comprehensive theory of $d$-invariants, of which $\dbot$ and $\dtop$ would be special (but important!) cases. We provide such a theory in this paper, making use of the action of the exterior algebra $\Lambda^*(H_1(Y)/\mathrm{torsion})$ to introduce {\em intermediate correction terms}  $d(Y,\spincs,V)$ for any subspace $V \subset H_1(Y)/\mathrm{torsion}$ and torsion \spinc\ structure $\spincs$, so long as $Y$ has standard $\HFi$ (which is true if the triple cup product on $H^1(Y;\Z)$ vanishes identically~\cite{lidman:hfinfinity}). We establish the basic properties of these invariants, including conjugation invariance (Proposition~\ref{prop:conjugation}), behavior under orientation reversal (Proposition~\ref{prop:duality}), additivity under connected sum (Proposition~\ref{prop:additivity}), and \spinc-rational homology cobordism invariance (Proposition~\ref{prop:QHcob}); we also show how they constrain the exoticness of the intersection forms of negative-definite $4$-manifolds bounded by a given $3$-manifold (Corollary~\ref{cor:intform}). Many of these results are adapted from~\cite{oz:boundary}, but the algebra involved is significantly trickier. We also give an application to link concordance in Section~\ref{S:links}.

\subsection*{Acknowledgements}
This work arose out of our joint project with Sa\v{s}o Strle~\cite{levine-ruberman-strle:surfaces}, to whom we are greatly indebted. We are also grateful to Peter Ozsv\'ath and Tye Lidman for helpful conversations.

\section{Algebraic preliminaries} \label{S:algebra}

We begin with some algebraic preliminaries concerning modules over exterior algebras.

Let $H$ be a finitely generated, free abelian group. Let $\Lambda = \Lambda^* H$ denote the exterior algebra of $H$ over $\Z$, which a graded-commutative ring. For any subspace $V \subset H$, let $\II^V V$ denote the (two-sided) ideal in $\Lambda^* H$ generated by $V$.

A \emph{$\Lambda$--module of homological type (resp.~cohomological type)} is a $\Q$--graded abelian group $M$ along with an action of $\Lambda$ on $M$ so that $\Lambda^i H$ takes the part in grading $r$ to the part in grading $r-i$ (resp.~$r+i$). We write the grading as a subscript ($M = \bigoplus_r M_r$) when $M$ is of homological type and as a superscript ($M = \bigoplus_r M^r$) when $M$ is of cohomological type. All $\Lambda$--modules will be assumed to be of either homological or cohomological type.

For any $\Lambda$--module $M$ (of homological type) and any subgroup $V \subset H$, let $\KK^V (M)$ denote the kernel of the action of $V$ on $M$, i.e.,
\[
\KK^V (M) = \{ x \in M \mid v \cdot x = 0 \ \forall v \in V\}.
\]
The part in grading $r$ is denoted $\KK^V_r (M)$. We write $\KK (M)$ for $\KK^H (M)$. Note that
\[
\KK^V (M) \cong \hom_\Lambda (\Lambda/\II^V, M).
\]
Similarly, let $\QQ^V (M)$ denote the quotient
\[
\QQ^V M = M / (\II^V \cdot M) = M \otimes_\Lambda \Lambda/\II^V,
\]
and again denote the part in grading $r$ by $\QQ^V_r(M)$. We write $\QQ (M)$ for $\QQ^H (M)$.

The exactness properties of Hom and tensor product imply that $\KK^V$ is a left-exact functor and $\QQ^V$ is a right-exact functor. Furthermore, the action of $\Lambda^* H$ on both $\KK^V (M)$ and $\QQ^V (M)$ descends to an action of $\Lambda^* ((H/V)/\Tors)$. We shall be particularly interested in $\QQ(\KK^V(M))$ and $\KK(\QQ^V(M))$, which we simply denote by $\QK^V(M)$ and $\KQ^V(M)$.

Likewise, for a module $M$ of cohomological type, we write $\KK_V(M)$ for the kernel and $\QQ_V(M)$ for the quotient, and indicate the grading with a superscript. For the rest of this section, we will implicitly use modules of homological type, but identical results hold for modules of cohomological type.

Let $M^*$ denote the dual of $M$, i.e.
\[
M^* = \hom_\Lambda (M, \Lambda).
\]
The part of $M^*$ in grading $r$ is simply $\hom_\Z(M_r,\Z)$, and the action of elements of $\Lambda$ is induced from the action of $\Lambda$ on $M$. If $M$ is of homological type, then $M^*$ is of cohomological type, and vice versa. The following lemma says that the functors $\QQ^V$ and $\KK^V$ are interchanged under taking duals:

\begin{lemma} \label{lemma:QKduality}
For any $\Lambda$--module $M$,
$(\QQ^V(M))^* \cong \KK_V(M^*)$ and $(\KK^V(M))^* \cong \QQ_V (M^*)$.
\end{lemma}

\begin{proof}
By basic properties of the Hom and tensor product functors, we have:
\begin{align*}
(\QQ^V(M))^* &\cong \hom_\Lambda (M \otimes_\Lambda \Lambda/\II^V, \Lambda) \\
&\cong \hom_\Lambda (\Lambda/\II^V, \hom_\Lambda(M,\Lambda))) \\
&\cong \KK_V (M^*)
\end{align*}
and, since $\Lambda/\II^V$ is finitely generated,
\begin{align*}
(\KK^V(M))^* &\cong \hom_\Lambda (\hom_\Lambda(\Lambda/\II^V, M), \Lambda) \\
&\cong \hom_\Lambda (M, \Lambda/\II^V) \\
&\cong \hom_\Lambda (M, \Lambda) \otimes_\Lambda \Lambda/\II^V \\
&\cong \QQ_V (M^*). \qedhere
\end{align*}
\end{proof}

Next, suppose that $H = H_2 \oplus H_2$. Write $\Lambda_1 = \Lambda^*(H_1)$ and $\Lambda_2 = \Lambda^*(H_2)$, so that $\Lambda = \Lambda_1 \otimes_\Z \Lambda_2$. If $M_1$ and $M_2$ are modules over $\Lambda_1$ and $\Lambda_2$, respectively, then $M_1 \otimes_\Z M_2$ is a module over $\Lambda$.

\begin{lemma} \label{lemma:directsum}
For any modules $M_1$ over $\Lambda_1$ and $M_2$ over $\Lambda_2$, and subgroups $V_1 \subset H_1$ and $V_2 \subset H_2$, \[
\KK^{V_1 \oplus V_2}(M_1 \otimes_\Z M_2) \cong \KK^{V_1}(M_1) \otimes_\Z \KK^{V_2}(M_2)
\]
and
\[
\QQ^{V_1 \oplus V_2}(M_1 \otimes_\Z M_2) \cong \QQ^{V_1}(M_1) \otimes_\Z \QQ^{V_2}(M_2).
\]
\end{lemma}

\begin{proof}
Note that $\Lambda / \II^{V_1 \oplus V_2} \cong (\Lambda_1/\II^{V_1}) \otimes_\Z (\Lambda_2/\II^{V_2})$. Therefore, we have:
\begin{align*}
\KK^{V_1 \oplus V_2}(M_1 \otimes_\Z M_2) &\cong \hom_{\Lambda} (M_1 \otimes_\Z M_2, \Lambda / \II^{V_1 \oplus V_2} ) \\
&\cong \hom_{\Lambda_1 \otimes_\Z \Lambda_2} (M_1 \otimes_\Z M_2, (\Lambda_1/\II^{V_1}) \otimes_\Z (\Lambda_2/\II^{V_2}) )\\
&\cong \hom_{\Lambda_1} (M_1, \Lambda_1/\II^{V_1}) \otimes_\Z \hom_{\Lambda_2} (M_2, \Lambda_2/\II^{V_2}) \\
&\cong \KK^{V_1}(M_1) \otimes_\Z \KK^{V_2}(M_2)
\end{align*}
and
\begin{align*}
\QQ^{V_1 \oplus V_2}(M_1 \otimes_\Z M_2) &\cong (M_1 \otimes_\Z M_2) \otimes_{\Lambda} (\Lambda / \II^{V_1 \oplus V_2} ) \\
&\cong (M_1 \otimes_\Z M_2) \otimes_{\Lambda_1 \otimes_\Z \Lambda_2} ( (\Lambda_1/\II^{V_1}) \otimes_\Z (\Lambda_2/\II^{V_2}) ) \\
&\cong (M_1 \otimes_{\Lambda_1} \Lambda_1/\II^{V_1} ) \otimes_\Z ( M_2 \otimes_{\Lambda_2} \Lambda_2/\II^{V_2} ) \\
&\cong \QQ^{V_1}(M_1) \otimes_\Z \QQ^{V_2}(M_2). \qedhere
\end{align*}
\end{proof}

Next, let $H^*$ be the dual of $H$, which is isomorphic to $H$ (non-canonically) since $H$ is finite dimensional, and let $M^{\st} = \Lambda^* H^* \otimes \Z[U,U^{-1}]$, graded so that $\Lambda^i H^* \otimes U^n$ is in grading $i - 2n$. This is naturally a module over $\Lambda \otimes \Z[U]$ of homological type, where, for $v \in \Lambda^*H$ and $\lambda \in \Lambda^* H^*$, we have
\[
(v \otimes U^k) \cdot (\lambda \otimes U^l) = i_v (\lambda) \otimes U^{k+l}
\]
where $i_v$ denotes contraction by $v$. If $\rank H = n > 0$, then $M_i \cong \Z^{2^{n-1}}$ for every $i \in Z$.

A subgroup $V \subset H$ is called \emph{primitive} if the quotient $H/V$ is free abelian. The following lemmas concerning $M^{\st}$ are easy to verify:

\begin{lemma} \label{lemma:lift}
If $v_1, \dots, v_k$ are elements of a basis for $H$, and $x \in M^{\st}_i$ satisfies $(v_1 \wedge \dots \wedge v_k) \cdot x = 0$, then there exists $x' \in M_{i+k}$ such that $(v_1 \wedge \dots \wedge v_k) \cdot x' = x$.
\end{lemma}

\begin{lemma} \label{lemma:tower}
For any primitive subgroup $V$ of a free abelian group $H$, we have
\[
\QK^V (M^{\st}) \cong \KQ^V (M^{\st}) \cong \Z[U,U^{-1}].
\]
\end{lemma}

\section{Correction terms for manifolds with standard \texorpdfstring{$\HFi$}{HF-infinity}} \label{S:correction}

Before defining the correction terms, we recall a few facts about the invariant $\HFi$. (We assume that the reader is familiar with the definition of Heegaard Floer homology \cite{oz:hf}.)

Let $Y$ be a closed, oriented $3$-manifold. We write $H_1^T(Y)$ for $H_1(Y;\Z)/\Tors$. Note that $H_1^T(Y)$ and $H^1(Y)$ are canonically dual to one another.

\begin{definition} \label{def:standardHFi}
Let $Y$ be a closed, oriented $3$-manifold. We say that $\HFi(Y)$ is \emph{standard} if for each torsion spin$^c$ structure $\spincs$ on $Y$, we have
\begin{equation} \label{eq:standardHFi}
\HFi(Y,\spincs) \cong \Lambda^* H^1(Y;\Z) \otimes \Z[U,U^{-1}]
\end{equation}
as a relatively graded $\Lambda^*(H_1^T/\Tors) \otimes \Z[U]$--module.
\end{definition}

The motivation for this definition is that for any $Y$, there is a spectral sequence whose $E^2$ term is the right-hand side of \eqref{eq:standardHFi} and which converges to $\HFi(Y, \spincs)$. If this spectral sequence collapses at the $E^2$ page, then the isomorphism \eqref{eq:standardHFi} holds, at least on the level of groups. Ozsv\'ath and Szab\'o conjectured that the only possibly nontrivial differential in the spectral sequence is $d^3$ and gave a conjectural formula for this differential in terms of the triple cup product form
\[
H^1(Y;\Z) \otimes H^1(Y;\Z) \otimes H^1(Y;\Z) \to \Z
\]
given by
\[
\alpha \otimes \beta \otimes \gamma \mapsto \gen{\alpha \cup \beta \cup \gamma, [Y]}.
\]
(See also \cite{mark:triple}.)

Lidman \cite{lidman:hfinfinity} recently proved that a certain variant of $\HFi$ (with coefficients in $\Z/2\Z$, completed with respect to $U$) is completely determined by the triple cup product form on $Y$. When this form vanishes, a modified version of his proof goes through for ordinary $\HFi$:

\begin{theorem} \label{thm:Lidman}
If $Y$ is a closed oriented $3$-manifold such that the triple cup product form vanishes identically, then $\HFi(Y)$ is standard.
\end{theorem}

\begin{proof}
Let $n = b_1(Y)$. According to \cite[Proposition 2.4 and Theorem 3.1]{lidman:hfinfinity}, we may find nonzero integers $m_1, \dots, m_k$ such that $Z = Y \conn L(m_1, 1) \conn \cdots \conn L(m_k, 1)$ is presented by integer surgery on a framed link $\L = \K_1 \cup \dots \cup \K_n \subset \conn^n S^1 \times S^2$, where each of the framed knots $\K_i$ is nulhomologous and has nonzero framing, and the pairwise linking numbers of $\L$ are all zero. Let $Z_0 = \conn^n S^1 \times S^2$, and for $i=1, \dots, n$, let $Z_i = Z_0(\K_1, \dots, \K_i)$ and let $W_i$ be the two-handle cobordism from $Z_{i-1}$ to $Z_i$. We inductively show that $\HFi(Z_i)$ is standard. The base case $i=0$ is clear. For the induction step, assume that $\HFi(Z_{i-1})$ is standard. If $W_i$ is negative-definite, then \cite[Proposition 9.4]{oz:boundary} says for any any spin$^c$ structure $\spincs$ on $W_i$ whose restriction to both ends is torsion, the induced map $F^\infty_{W_i,\spincs} \co \HFi(Z_{i-1}, \spincs|_{Z_{i-1}}) \to \HFi(Z_i, \spincs|_{Z_i})$ is an isomorphism of $\Lambda^*(H_1/\Tors) \otimes \Z[U]$--modules, implying that $\HFi(Z_i)$ is standard. If $W_i$ is positive-definite, then $F^\infty_{-W_i,\spincs} \co \HFi(-Z_{i-1}, \spincs|_{Z_{i-1}}) \to \HFi(-Z_i, \spincs|_{Z_i})$ is an isomorphism. An oriented manifold $M$ has standard $\HFi$ if and only if $-M$ does, so $\HFi(Z_i)$ must be standard in this case as well, concluding the induction. Thus, $\HFi(Z)$ is standard, and the connected sum formula \cite[Theorem 6.2] {oz:hf-properties} then implies that $\HFi(Y)$ is standard as well.
\end{proof}

Suppose $\HFi(Y)$ is standard. Roughly speaking, we would like to define a $d$ invariant for each tower in $\HFi(Y,\s)$ --- i.e., the subspace $\lambda \otimes \Z[U,U^{-1}]$, where $\lambda$ is a homogeneous element of $\Lambda^* H^1(Y)$ --- as the minimal grading of nonzero elements in the image of that tower in $\HFp(Y,\spincs)$. The most important of these invariants are the ones associated to generators of the bottom and top exterior powers $\Lambda^0 H^1(Y)$ and $\Lambda^{b_1(Y)} H^1(Y)$. The problem with this approach is that the identification \eqref{eq:standardHFi} is not canonical (to the authors' knowledge), so $\lambda$ does not determine a subspace of $\HFi(Y,\spincs)$ in an invariant way. Instead, we may avoid the naturality issue by making use of the $H_1$ action to define invariants associated to every subspace of $H_1(Y)/\Tors$. For technical reasons, we actually need to define two rational numbers associated to each subspace, denoted $d(Y, \spincs, V)$ and $d^*(Y, \spincs, V)$. The $4$-dimensional applications that we describe below (Corollaries \ref{cor:QHS1xB3} and \ref{cor:intform}) only depend on knowing $d(Y, \spincs, V)$ for each $V$, but we need to use both collections of invariants to prove some of the key properties (Propositions \ref{prop:duality}, \ref{prop:additivity}, and \ref{prop:QHcob}) used in establishing these facts.

The short exact sequence
\[
0 \to \CFm(Y, \spincs) \xrightarrow{\iota} \CFi(Y, \spincs) \xrightarrow{\pi} \CFp(Y, \spincs) \to 0
\]
gives rise to a long exact sequence
\[
\cdots \to \HFm(Y,\spincs) \xrightarrow{\iota_*} \HFi(Y,\spincs) \xrightarrow{\pi_*} \HFp(Y,\spincs) \xrightarrow{\delta} \HFm(Y,\spincs) \to \cdots,
\]
which we denote by $\operatorname{HF}^\circ(Y,\spincs)$. Let $I^-(Y, \spincs) = \coker \delta = \im \iota_*$, and $I^+(Y, \spincs) = \ker \delta = \im \pi_*$, so that we have a short exact sequence
\begin{equation} \label{eq:Iexact}
0 \to I^-(Y, \spincs) \xrightarrow{\iota_*} \HFi(Y, \spincs) \xrightarrow{\pi_*} I^+(Y, \spincs) \to 0.
\end{equation}
Denote this short exact sequence by $I^\circ(Y, \spincs)$.

Since the functor $\KK^V$ is left-exact, we have an exact sequence
\begin{equation} \label{eq:Kleftexact}
0 \to \KK^V (I^-(Y, \spincs)) \xrightarrow{ \KK^V(\iota_*)} \KK^V (\HF^\infty(Y, \spincs)) \xrightarrow{ \KK^V(\pi_*)} \KK^V( I^+(Y,\spincs)),
\end{equation}
and hence, setting $J^+(Y, \spincs, V) = \im (\KK^V(\pi_*))$, a short exact sequence of modules over $\Lambda^* ((H_1^T(Y)/V)/\Tors)$,
\begin{equation} \label{eq:Kexact}
0 \to \KK^V (I^-(Y, \spincs)) \xrightarrow{ \KK^V(\iota_*)} \KK^V (\HF^\infty(Y, \spincs)) \xrightarrow{ \KK^V(\pi_*)} J^+(Y,\spincs,V) \to 0.
\end{equation}
Applying $\QQ$ to this sequence, we define $d(Y,\spincs,V)$ to be the minimal grading of any nontorsion element\footnote{Here and throughout, ``nontorsion'' simply means nontorsion as an element of an abelian group, without regard to the $\Lambda^* \otimes \Z[U]$--module structure.} in the image of
\begin{equation} \label{eq:QKpi}
\QK^V(\pi_*) \co \QK^V (\HF^\infty(Y, \spincs)) \to \QQ (J^+(Y,\spincs,V)).
\end{equation}

Similarly, if we apply $\QQ^V$ to $I^\circ$, we get an exact sequence
\begin{equation} \label{eq:Qrightexact}
\QQ^V (I^-(Y, \spincs)) \xrightarrow{ \QQ^V(\iota_*)} \QQ^V (\HF^\infty(Y, \spincs)) \xrightarrow{ \QQ^V(\pi_*)} \QQ^V( I^+(Y,\spincs)) \to 0,
\end{equation}
and hence, setting $J^-(Y, \spincs, V) = \im (\QQ^V(\iota_*))$, a short exact sequence
\begin{equation} \label{eq:Qexact}
0 \to J^-(Y, \spincs,V) \xrightarrow{ \QQ^V(\iota_*)} \QQ^V (\HF^\infty(Y, \spincs)) \xrightarrow{ \QQ^V(\pi_*)} \QQ^V(I^+(Y,\spincs,V)) \to 0.
\end{equation}
We define $d^*(Y,\spincs,V)$ to be the minimal grading of any nontorsion element in the image of
\begin{equation} \label{eq:KQpi}
\KQ^V(\pi_*) \co \KQ^V (\HF^\infty(Y, \spincs)) \to \KQ^V (I^+(Y,\spincs)).
\end{equation}

Finally, we define the \emph{bottom} and \emph{top correction terms} of $(Y, \spincs)$ to be
\[
\dbot(Y,\spincs) = d(Y, \spincs, H_1^T(Y)) = d^*(Y, \spincs, \{0\})
\]
and
\[
\dtop(Y,\spincs) = d(Y, \spincs, \{0\}) = d^*(Y, \spincs, H_1^T(Y)).
\]
The correction terms associated to subspaces of rank not equal to $0$ or $b_1(Y)$ are called the \emph{intermediate correction terms}.

We now prove some basic algebraic properties of the correction terms. To begin, the following lemma implies that it suffices to consider the correction terms associated to primitive subgroups of $H_1^T(Y)$.

\begin{lemma} \label{lemma:samerank}
If $V$ and $V'$ are subgroups of $H_1^T(Y)$ of the same rank such that $V' \subset V$, then $d(Y, \spincs, V) = d(Y, \spincs, V')$ and $d^*(Y, \spincs, V) = d^*(Y, \spincs, V')$.
\end{lemma}

\begin{proof}
Since $V / V'$ is torsion, there exists an $n$ such that for any $v \in V$, $nv \in V'$.

Since $\HFi(Y,\spincs)$ is torsion-free, $\KK^V (\HFi(Y, \spincs)) = \KK^{V'} (\HFi(Y,\spincs))$, while $\KK^V I^+(Y, \spincs) \subset \KK^{V'} I^+(Y, \spincs)$.

Suppose $\xi \in \KK^V (\HFi(Y,\spincs))$ is in grading $d(Y, \spincs, V)$, and the class of $\xi$ in $\QK^V (\HFi(Y, \spincs))$ maps to a nontorsion element of $\QQ (J^+(Y, \spincs, V))$. We claim that no nonzero multiple of $\pi_*(\xi)$ is in the image of the action of $H_1^T(Y) / V'$. Specifically, suppose that
\[
\pi(m \xi) = \sum_{i=1}^k w_i \cdot \eta_i,
\]
where $w_1, \dots, w_k$ are elements of $H_1^T(Y)$ that are linearly independent in $H_1^T(Y) / V'$, hence in $H_1^T(Y)/V$, and $\eta_1, \dots, \eta_k \in J^+(Y, \spincs, V')$. Since $n \eta_i \in J^+(Y, \spincs, V)$, we see that $\pi(m n \xi)$ is in the image of the action of $H_1^T(Y) / V$ on $J^+(Y, \spincs, V)$, a contradiction. Thus, $d(Y, \spincs, V') \le d(Y, \spincs, V)$. A similar but shorter argument shows that $d(Y, \spincs, V) \le d(Y, \spincs, V')$.

The statement for $d^*$ is left to the reader as an exercise.
\end{proof}

Viewed another way, Lemma \ref{lemma:samerank} says that $d(Y, \spincs, V)$ is really an invariant of the subspace $V \otimes \Q \subset H_1(Y;\Q)$, meaning that $d(Y, \spincs, \cdot)$ is really a function on the Grassmannian of $H_1(Y;\Q)$.

Next, we consider nested subspaces of different ranks.

\begin{proposition} \label{prop:rank}
If $V$ and $V'$ are subspaces of $H_1^T(Y)$ of ranks $k$ and $l$ respectively, and $V' \subset V$, then
\[
d(Y, \spincs, V) \ge d(Y, \spincs, V') - \rank(V/V')
\]
and
\[
d^*(Y, \spincs, V) \le d^*(Y, \spincs, V') + \rank(V/V').
\]
In particular,
\[
\dbot(Y, \spincs) \ge \dtop(Y,\spincs) - b_1(Y).
\]
Moreover, in each of these inequalities, the two sides are congruent modulo $2\Z$.
\end{proposition}

\begin{proof}
It suffices to prove the result when $\rank V = \rank V' + 1$; the general result then follows by induction. Furthermore, by Lemma \ref{lemma:samerank}, we may assume that $V/V' \cong \Z$. Let $k = \rank V$; we may find a basis $v_1, \dots, v_k$ for $V$ such that $v_1, \dots, v_{k-1}$ is a basis for $V'$.

Let $\xi \in \KK^V_d (\HFi(Y,\spincs))$, where $d = d(Y, \spincs, V)$, such that $\pi_*(\xi)$ represents a nontorsion class in $\QQ(J^+(Y, \spincs, V))$. There exists an element $\eta \in \HFi_{d+k}(Y, \spincs)$ such that $(v_1 \wedge \cdots \wedge v_k) \cdot \eta = \xi$. Let $\zeta = (v_1 \wedge \cdots \wedge v_{k-1}) \cdot \eta$; note that $\zeta \in \KK^{V'}_{d+1}(\HFi(Y, \spincs))$.

We claim that $\pi_*(\zeta) \in \KK^{V'} (I^+(Y, \spincs)$ represents a nontorsion class in $\QQ(J^+(Y, \spincs,V'))$, which will imply (by Lemma \ref{lemma:tower}) that
\[
d(Y, \spincs, V') = \gr(\zeta) - 2j
\]
for some nonnegative integer $j$, as required. Suppose, toward a contradiction, that
\[
n \pi_*(\zeta) = \sum_{i=1}^m w_i \cdot \pi_*(\theta_i).
\]
for some $w_1, \dots, w_m \in H_1^T(Y) \minus V'$ and some $\theta_1, \dots, \theta_m \in \KK^{V'}_{d+2}( \HFi(Y, \spincs))$ (meaning that $\pi_*(\theta_i) \in J^+(Y,\spincs,V')$). Then
\[
\sum_{i=1}^m w_i \cdot (-v_k \cdot \pi_*(\theta_i)) = v_k \cdot n \pi_*(\zeta) = n \pi_*(\xi).
\]
Since $-v_k \cdot \pi_*(\theta_i) \in J^+(Y,\spincs,V')$, this contradicts the fact that $\pi_*(\xi)$ represents a nontorsion class in $\QQ(J^+(Y,\spincs,V))$.

The proof for $d^*$ proceeds similarly; the statement also follows from Proposition \ref{prop:duality}, below.
\end{proof}

\begin{corollary} \label{cor:simple}
If $\dbot(Y, \spincs) = \dtop(Y, \spincs) - b_1(Y)$, then for any subgroup $V \subset H_1^T(Y)$, we have
\[
d(Y, \spincs, V) = \dtop(Y, \spincs) - \rank(V).
\]
and
\[
d^*(Y, \spincs, V) = \dbot(Y, \spincs) + \rank(V).
\]
(We say that the correction terms of $(Y,\spincs)$ are \emph{simple} in this case.)
\end{corollary}

\begin{proof}
Apply Proposition \ref{prop:rank} to the inclusions $\{0\} \subset V$ and $V \subset H_1^T(Y)$.
\end{proof}

Just as in the rational homology sphere case, the correction terms may also be formulated in terms of the map $\iota_* \co \HFm(Y,\spincs) \to \HFi(Y,\spincs)$. Let $d^-(Y,\spincs,V)$ be the maximal grading of an element of $\QK^V I^-(Y, \spincs)$ whose image in $\QK^V \HFi(Y, \spincs)$ is nontrivial, and let $d^{*-}(Y, \spincs, V)$ be the maximal grading of an element of $\KK J^-(Y,\spincs,V)$ whose image in $\KQ^V \HFi(Y, \spincs)$ is nontorsion (see \eqref{eq:Qexact}). The exact sequences obtained by applying $\QQ$ to \eqref{eq:Kexact} and $\KK$ to \eqref{eq:Qexact}, together with Lemma \ref{lemma:tower}, imply that $d^-(Y,\spincs,V) = d(Y,\spincs,V) -2$ and $d^{*-}(Y,\spincs,V) = d^*(Y, \spincs, V) -2$. We shall make use of these reformulations below.

Before moving on to further properties of the correction terms, we consider two examples.

\begin{figure}
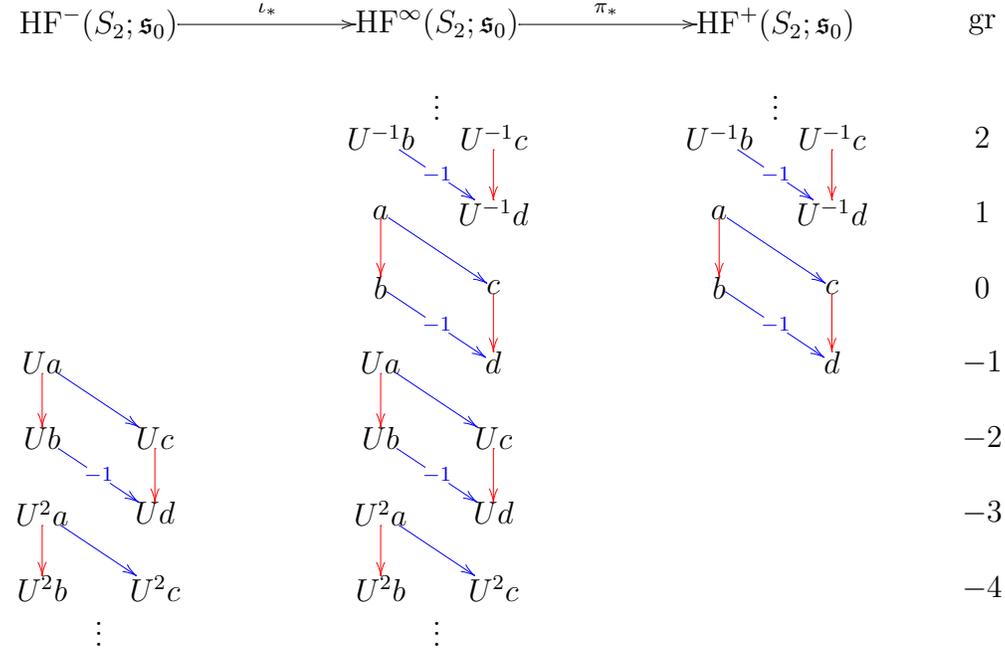

\[
\xy
(0,0)*{Ua}="Ua-"; (0,-10)*{Ub}="Ub-"; (15,-10)*{Uc}="Uc-"; (15,-20)*{Ud}="Ud-";
(0,-20)*{U^2a}="U2a-"; (0,-30)*{U^2b}="U2b-"; (15,-30)*{U^2c}="U2c-"; 
(7.5,-35)*{\vdots};
(7.5,45)*{\HFm(S_2;\spincs_0)}="HFm";
(52.5,35)*{\vdots};
(45,30)*{U^{-1}b}="U-1b"; (60,30)*{U^{-1}c}="U-1c"; (60,20)*{U^{-1}d}="U-1d";
(45,20)*{a}="a"; (45,10)*{b}="b"; (60,10)*{c}="c"; (60,0)*{d}="d";
(45,0)*{Ua}="Ua"; (45,-10)*{Ub}="Ub"; (60,-10)*{Uc}="Uc"; (60,-20)*{Ud}="Ud";
(45,-20)*{U^2a}="U2a"; (45,-30)*{U^2b}="U2b"; (60,-30)*{U^2c}="U2c"; 
(52.5,-35)*{\vdots};
(52.5,45)*{\HFi(S_2;\spincs_0)}="HFi";
(97.5,35)*{\vdots};
(90,30)*{U^{-1}b}="U-1b+"; (105,30)*{U^{-1}c}="U-1c+"; (105,20)*{U^{-1}d}="U-1d+";
(90,20)*{a}="a+"; (90,10)*{b}="b+"; (105,10)*{c}="c+"; (105,0)*{d}="d+";
(97.5,45)*{\HFp(S_2;\spincs_0)}="HFp";
(125,45)*{\gr};
(125,30)*{2};
(125,20)*{1};
(125,10)*{0};
(125,0)*{-1};
(125,-10)*{-2};
(125,-20)*{-3};
(125,-30)*{-4};
{\ar@[red] "Ua-";"Ub-"};
{\ar@[red] "Uc-";"Ud-"};
{\ar@[red] "U2a-";"U2b-"};
{\ar@[blue] "Ua-";"Uc-"};
{\ar@[blue]|{{\color{blue} -1}} "Ub-";"Ud-"};
{\ar@[blue] "U2a-";"U2c-"};
{\ar@[red] "U-1c";"U-1d"};
{\ar@[red] "a";"b"};
{\ar@[red] "c";"d"};
{\ar@[red] "Ua";"Ub"};
{\ar@[red] "Uc";"Ud"};
{\ar@[red] "U2a";"U2b"};
{\ar@[blue]|{{\color{blue} -1}} "U-1b";"U-1d"};
{\ar@[blue] "a";"c"};
{\ar@[blue]|{{\color{blue} -1}} "b";"d"};
{\ar@[blue] "Ua";"Uc"};
{\ar@[blue]|{{\color{blue} -1}} "Ub";"Ud"};
{\ar@[blue] "U2a";"U2c"};
{\ar@[red] "U-1c+";"U-1d+"};
{\ar@[red] "a+";"b+"};
{\ar@[red] "c+";"d+"};
{\ar@[blue]|{{\color{blue} -1}} "U-1b+";"U-1d+"};
{\ar@[blue] "a+";"c+"};
{\ar@[blue]|{{\color{blue} -1}} "b+";"d+"};
{\ar^{\iota_*} "HFm";"HFi"};
{\ar^{\pi_*} "HFi";"HFp"};
\endxy
\]
\caption{The short exact sequence $\HFcirc(S_2, \spincs_0)$. The red (vertical) arrows indicate the action of $\alpha$, while the blue (diagonal) arrows indicate the action of $\beta$.} \label{fig:S1xS2}
\end{figure}

\begin{example}\label{ex:S1xS2}
For any $n \ge 0$, consider the manifold $S_n = \conn^n S^1 \times S^2$, with its unique torsion spin$^c$ structure $\spincs_0$. As shown by Ozsv\'ath and Szab\'o, the group $\HF^{\le 0}(S_n, \spincs_0) \subset \HFi(S_n, \spincs_0)$ has a canonical top-dimensional generator (up to sign), which we denote by $\Theta_n$, in grading $n/2$. It is well-known that $\Theta_n$ generates $\HF^{\le 0}(S_n, \spincs_0)$ as a $\Lambda^* H_1(S_n) \otimes \Z[U]$--module, and it generates $\HFi(S_n, \spincs_0)$ as a $\Lambda^* H_1(S_n) \otimes \Z[U,U^{-1}]$--module. Then $\pi_*(\Theta_n)$ is a nontorsion class in $\QK^{\{0\}}(I^+(S_n, \spincs_0))$, while $\pi_*(U \Theta_n)=0$, meaning that $\dtop(S_n, \spincs_0) = n/2$. Moreover, if $\Delta$ is a generator of $\Lambda^n H_1(S_n)$, then $\pi_*(\Delta \Theta_n)$ is a nontorsion class in $\KQ^{\{0\}}$, so $\dbot(S_n, \spincs_0) = -n/2$. Hence the correction terms of $(S_n,\spincs_0)$ are simple and are given by Corollary \ref{cor:simple}:
\begin{equation} \label{eq:S1xS2}
d(S_n, \spincs_0, V) = \frac{n}{2} - \rank V \quad \text{and} \quad d^*(S_n, \spincs_0, V) = -\frac{n}{2} + \rank V.
\end{equation}

We illustrate the case where $n=2$ in Figure~\ref{fig:S1xS2}. Let $\alpha, \beta$ denote a basis for $H_1(S_2)$, and set $a = \Theta_2$, $b = \alpha \cdot \Theta_2$, $c = \beta \cdot \Theta_2$, and $d = (\alpha \wedge \beta) \cdot \Theta_2$.
Figure \ref{fig:S1xS2} depicts the short exact sequence relating $\HFm(S_2)$, $\HFi(S_2)$, and $\HFp(S_2)$. (Even though $U$ is only invertible on $\HFi$, we may label elements of $\HFp$ using negative powers of $U$ without confusion.)
\end{example}

\begin{figure}
\[
\xy
(15,-10)*{Uc}="Uc-"; (15,-20)*{Ud}="Ud-";
(0,-20)*{U^2a}="U2a-"; (0,-30)*{U^2b}="U2b-"; (15,-30)*{U^2c}="U2c-"; 
(7.5,-35)*{\vdots};
(7.5,45)*{\HFm(Y;\spincs_0)}="HFm";
(52.5,35)*{\vdots};
(45,30)*{U^{-1}b}="U-1b"; (60,30)*{U^{-1}c}="U-1c"; (60,20)*{U^{-1}d}="U-1d";
(45,20)*{a}="a"; (45,10)*{b}="b"; (60,10)*{c}="c"; (60,0)*{d}="d";
(45,0)*{Ua}="Ua"; (45,-10)*{Ub}="Ub"; (60,-10)*{Uc}="Uc"; (60,-20)*{Ud}="Ud";
(45,-20)*{U^2a}="U2a"; (45,-30)*{U^2b}="U2b"; (60,-30)*{U^2c}="U2c"; 
(52.5,-35)*{\vdots};
(52.5,45)*{\HFi(Y;\spincs_0)}="HFi";
(97.5,35)*{\vdots};
(90,30)*{U^{-1}b}="U-1b+"; (105,30)*{U^{-1}c}="U-1c+"; (105,20)*{U^{-1}d}="U-1d+";
(90,20)*{a}="a+"; (90,10)*{b}="b+"; (105,10)*{c}="c+"; (105,0)*{d}="d+";
(90,0)*{Ua}="Ua+"; (90,-10)*{Ub}="Ub+";
(97.5,45)*{\HFp(Y;\spincs_0)}="HFp";
(125,45)*{\gr};
(125,30)*{2};
(125,20)*{1};
(125,10)*{0};
(125,0)*{-1};
(125,-10)*{-2};
(125,-20)*{-3};
(125,-30)*{-4};
{\ar@[red] "Uc-";"Ud-"};
{\ar@[red] "U2a-";"U2b-"};
{\ar@[blue] "U2a-";"U2c-"};
{\ar@[red] "U-1c";"U-1d"};
{\ar@[red] "a";"b"};
{\ar@[red] "c";"d"};
{\ar@[red] "Ua";"Ub"};
{\ar@[red] "Uc";"Ud"};
{\ar@[red] "U2a";"U2b"};
{\ar@[blue]|{{\color{blue} -1}} "U-1b";"U-1d"};
{\ar@[blue] "a";"c"};
{\ar@[blue]|{{\color{blue} -1}} "b";"d"};
{\ar@[blue] "Ua";"Uc"};
{\ar@[blue]|{{\color{blue} -1}} "Ub";"Ud"};
{\ar@[blue] "U2a";"U2c"};
{\ar@[red] "U-1c+";"U-1d+"};
{\ar@[red] "a+";"b+"};
{\ar@[red] "c+";"d+"};
{\ar@[blue]|{{\color{blue} -1}} "U-1b+";"U-1d+"};
{\ar@[blue] "a+";"c+"};
{\ar@[blue]|{{\color{blue} -1}} "b+";"d+"};
{\ar@[red] "Ua+";"Ub+"};
{\ar^{\iota_*} "HFm";"HFi"};
{\ar^{\pi_*} "HFi";"HFp"};
\endxy
\]
\caption{The short exact sequence $\HF^\circ(Y,\spincs_0)$, where $Y = S^1 \times S^2 \conn S^3_0(K)$ and $K$ is the right-handed trefoil.} \label{fig:hyp}
\end{figure}

\begin{example} \label{ex:hyp}
Let $K \subset S^3$ denote the right-handed trefoil. The Heegaard Floer homology of $S^3_0(K)$ in the torsion spin$^c$ structure $\spincs_0$ is computed in \cite[Section 8.1]{oz:boundary}:
\[
\HFp_k(S^3_0(K), \spincs_0) =
\begin{cases}
\Z & \text{if } k \equiv 1/2 \pmod 2 \text{ and } k \ge -3/2 \\
\Z & \text{if } k \equiv -1/2 \pmod 2 \text{ and } k \ge -1/2 \\
0 & \text{otherwise}
\end{cases}
\]
with the action of a generator of $H_1(S^3_0(K))$ taking $\HFp_k(S^3_0(K))$ isomorphically to $\HFp_{k-1}(S^3_0(K))$ if $k \equiv 1/2 \pmod 2$ and $k \ge 1/2$. Since $\HFi(S^3_0(K)) \to \HFp(S^3_0(K))$ is surjective, it is easy to see that $\dbot(S^3_0(K)) = -1/2$ and $\dtop(S^3_0(K))=-3/2$.

Now let $Y = S^1 \times S^2 \conn S^3_0(K)$. The connected sum formula \cite{oz:hf-properties} implies that $\HF^\circ(Y, \spincs_0)$ is shown in Figure \ref{fig:hyp}, where $\alpha$ is a generator of $H_1(S^1 \times S^2)$ and $\beta$ is a generator of $H_1(S^3_0(K))$. (That is, the generators of $\HFi(Y, \spincs_0)$ are named just as in the previous example, but now $\pi_*(Ua)$ and $\pi_*(Ub)$ are nontrivial.) We compute the $d$ invariants of $Y$.

First, if $H = H_1^T(Y)$, then $\KK^H \HFi(Y, \spincs_0) = \QK^H \HFi(Y, \spincs_0)$ is generated by $\{U^i d \mid i \in \Z\}$; $\KK^H \HFp(Y, \spincs_0)$ is generated by $\{U^i d \mid i \le -1\} \cup \{U b\}$; and $J^+(Y, \spincs_0, H) = \QQ (J^+(Y, \spincs_0,H))$ is generated by $\{U^i d \mid i \le -1\}$.\footnote{This example shows that the final map in the exact sequence \eqref{eq:Kleftexact} need not be surjective, necessitating the definition of $J^+$.} Thus,
\[
\dbot(Y, \spincs_0) = d(Y, \spincs_0, H) = -1.
\]
Similarly, $\QQ^H \HFi(Y, \spincs_0) = \KQ^H \HFi(Y, \spincs_0)$ is generated by the classes of $\{U^i a \mid i \in \Z\}$, and $\QQ^H I^+ (Y, \spincs_0) = \KQ^H I^+(Y, \spincs_0)$ is generated by the classes of $\{U^i a \mid i \le 1\}$, which means that
\[
\dtop(Y, \spincs_0) = d^*(Y, \spincs_0, H) = -1
\]
as well.

Next, suppose $V$ is a primitive rank-$1$ subspace of $H_1(Y)$, generated by $p \alpha + q \beta$ for $p,q$ relatively prime. Then $\KK^V \HFi(Y, \spincs_0)$ is generated by $\{U^i d, U^i(pb+qc) \mid i \in \Z\}$, with the action of a generator of $H_1(Y)/V$ taking $U^i(pb+qc)$ to $\pm U^i d$, so $\QK^V \HFi(Y, \spincs_0)$ is generated by the classes of $U^i (pb+qc)$ for all $i \in \Z$. Likewise, $J^+(Y, \spincs_0, V)$ is generated by $\{U^i d, U^i(pb+qc) \mid i \le 0\}$, together with $pUb$ provided $p \ne 0$, and the action of $H_1(Y)/V$ is just like in $\KK^V \HFi(Y, \spincs_0)$. Thus, $\QQ^V J^+(Y, \spincs_0, V)$ is generated by $\{U^i(pb+qc) \mid i \le -1\}$, together with $b$ provided that $p \ne 0$. It follows that
\begin{equation} \label{eq:dhyp}
d(Y, \spincs_0, \gen{p\alpha + q\beta}) = \begin{cases} 0 & p = 0 \\ -2 & p \ne 0. \end{cases}
\end{equation}

The quotient $\QQ^V (\HFi(Y, \spincs_0))$ can be generated by the classes of $U^i a$ and $U^i(rb + sc)$ for all $\mid i \in \Z$, where $(r,s)$ is any pair of integers satisfying $ps-qr = 1$. The classes of $U^i(rb+sc)$ generate $\KQ^V (\HFi(Y, \spincs_0))$. Likewise, $\QQ^V (I^+ (Y, \spincs_0))$ is generated by the classes of $U^i a$ for $i \le -1$ and $U^i (rb+sc)$ for all $i \le 0$, along with a $\Z/p$ factor generated by $Ub$. If $p \ne 0$, then the element of least grading in $\KQ^V( \HFi(Y, \spincs_0)$ that maps to a nontorsion element of $\KQ^V (I^+(Y, \spincs_0))$ is $rb+sc$; if $p=0$, it is $U(rb+rc)$. Thus, we have
\begin{equation} \label{eq:d*hyp}
d^*(Y, \spincs_0, \gen{p\alpha + q\beta}) = \begin{cases} -2 & p = 0 \\ 0 & p \ne 0. \end{cases}
\end{equation}
\end{example}

%

\section{Properties of the correction terms} \label{S:properties}

We now prove that the correction terms satisfy conjugation invariance, orientation reversal, and connected sum formulas, and study the behavior of the correction terms under negative definite cobordisms. Most of these results are analogous to those in \cite[Sections 4 and 9]{oz:boundary}.

\begin{proposition}[Conjugation invariance] \label{prop:conjugation}
Let $Y$ be a closed, oriented 3-manifold with standard $\HFi$, let $\spincs$ be a torsion spin$^c$ structure on $Y$, and let $\bar \spincs$ denote the conjugate spin$^c$ structure. For any subspace $V \subset H_1^T(Y)$, we have $d(Y, \spincs, V) = d(Y, \bar\spincs, V)$ and $d^*(Y, \spincs, V) = d^*(Y, \bar\spincs, V)$.
\end{proposition}

\begin{proof}
This follows immediately from the conjugation invariance of Heegaard Floer homology.
\end{proof}

\begin{proposition}[Orientation reversal] \label{prop:duality}
Let $Y$ be a closed, oriented 3-manifold with standard $\HFi$, and let $\spincs$ be a torsion spin$^c$ structure on $Y$. For any subspace $V \subset H_1^T(Y)$, we have $d(Y, \spincs,V) = -d^*(-Y, \spincs,V)$ and $d^*(Y, \spincs,V) = -d^*(-Y, \spincs,V)$.
\end{proposition}

\begin{proof}
We adapt the proof of \cite[Proposition 4.2]{oz:boundary}. By interchanging the roles of $Y$ and $-Y$, it suffices to prove the first statement. Furthermore, by Lemma \ref{lemma:samerank}, we may assume that $V$ is a primitive subspace. Let $k = \rank V$ and $b = b_1(Y)$.

As in \cite{oz:boundary}, we denote the Heegaard Floer cohomology groups of $(-Y,\spincs)$ by $\HF_-(-Y,\spincs)$, $\HF_\infty(-Y,\spincs)$, and $\HF_+(-Y,\spincs)$, indicating the absolute grading with a superscript. Note that the action of $H_1^T(Y)$ on Heegaard Floer cohomology increases grading by $1$. There are duality isomorphisms $\DD^-$, $\DD^\infty$, $\DD^+$, which fit into a commutative diagram whose rows are exact:
\[
\xymatrix{
\cdots \ar[r]^-{\delta} & \HFm_m(Y,\spincs) \ar[r]^{\iota_*} \ar[d]^{\DD^-}_{\cong} & \HFi_m (Y,\spincs) \ar[r]^{\pi_*} \ar[d]^{\DD^\infty}_{\cong} & \HFp_m(Y,\spincs) \ar[r]^-{\delta} \ar[d]^{\DD^+}_{\cong} & \cdots
 \\
\cdots \ar[r]^-{\delta^*} & \HF_+^{-m-2}(-Y,\spincs) \ar[r]^{\pi^*} & \HF_\infty^{-m-2}(-Y,\spincs) \ar[r]^{\iota^*} & \HF_-^{-m-2}(-Y,\spincs) \ar[r]^-{\delta^*} & \cdots
}
\]
Let $I_+(-Y, \spincs) = \coker \delta^*$ and $I_-(-Y, \spincs) = \ker \delta^*$, which fit into a short exact sequence \begin{equation} \label{eq:I*exact}
0 \to I_+(-Y, \spincs) \xrightarrow{\pi^*} \HF_\infty(-Y, \spincs) \xrightarrow{\iota^*} I_-(-Y, \spincs) \to 0.
\end{equation}
Moreover, the duality isomorphisms commute with the $H_1$ action. As a result, the isomorphisms $\DD^\circ$ descend to isomorphisms on kernels and quotients of the $H_1$ action. Specifically, if we let $J_-(-Y, \spincs, V)$ denote the image of the map
\[
\KK^V(\iota^*)\co \KK^V( \HF_\infty (-Y,\spincs) ) \to \KK^V (I_-(-Y, \spincs)),
\]
we obtain, for any $m \in \Q$, a commutative diagram
\begin{equation} \label{eq:QKdual}
\xymatrix@C0.6in{
\QK^V_m( \HFi(Y,\spincs) ) \ar[r]^{\QK^V_m(\pi_*)} \ar[d]^{\DD^\infty}_\cong & \QQ_m( J^+(Y, \spincs, V)) \ar[d]^{\DD^+}_\cong \\
\QK^V_{-m-2}( \HF_\infty (-Y,\spincs) ) \ar[r]^{\QK^V_m(\iota^*)} & \QQ_{-m-2} (J_-(-Y, \spincs, V)).
}
\end{equation}

By Lemma \ref{lemma:QKduality}, the nonsingular pairings
\[
\HF_\infty(-Y, \spincs) \otimes \HFi(-Y, \spincs) \to \Z
\]
and
\[
\HF_-(-Y, \spincs) \otimes \HFm(-Y, \spincs) \to \Z
\]
descend, via Lemma \ref{lemma:QKduality}, to nonsingular pairings
\[
\QK^V( \HF_\infty (-Y,\spincs) ) \otimes \KQ^V (\HFi(-Y, \spincs)) \to \Z
\]
and
\[
\QQ_{l} (J_-(-Y, \spincs, V)) \otimes \KK_l(J^-(-Y, \spincs,V)) \to \Z.
\]
Therefore, for any $l \in \Q$, the map
\begin{equation} \label{eq:QKiota*}
\QK^V_l(\iota^*) \co \QK^V_{l}( \HF_\infty (-Y,\spincs) ) \to \QQ_{l} (J_-(-Y, \spincs, V))
\end{equation}
contains nontorsion elements in its image if and only if the map
\begin{equation} \label{eq:KQiota}
\KQ^V_l(\iota_*) \co \KK_l(J^-(-Y, \spincs,V)) \to \KQ^V_l(\HFi(-Y, \spincs))
\end{equation}
does. As a result, we see (just as in the proof of \cite[Proposition 4.2]{oz:boundary}) that
\[
d(Y,\s,V) = -2 - d^{*-}(-Y, \spincs,V) = -d^*(-Y, \spincs, V),
\]
as required.
%
\end{proof}

\begin{proposition}[Additivity] \label{prop:additivity}
Let $Y$ and $Z$ be closed, oriented $3$-manifolds with standard $\HFi$, and let $\spinct$ and $\spincu$ be torsion spin$^c$ structures on $Y$ and $Z$ respectively. For any subspaces $V \subset H_1^T(Y)$ and $W \subset H_1^T(Z)$, we have \[
d(Y \conn Z, \spinct \conn \spincu, V \oplus W) = d(Y, \spinct, V) + d(Z, \spincu, W)
\]
and
\[
\dtop(Y \conn Z, \spinct \conn \spincu) = \dtop(Y, \spinct) + \dtop(Z, \spincu).
\]
\end{proposition}

\begin{proof}
By the connected sum formula for Heegaard Floer homology \cite[Theorem 6.2]{oz:hf-properties}, there are graded isomorphisms making the diagram
\begin{equation} \label{eq:connCF}
\xymatrix{
H_*(\CF^-(Y, \spincs) \otimes_{\Z[U]} \CF^- (Z, \spinct))[2] \ar[r]^-{F^-_{Y \conn Z}}_-{\cong} \ar[d]^{(\iota^Y \otimes \iota^Z)_*}  &
  \HF^-(Y \conn Z, \spincs \conn \spinct) \ar[d]^{\iota^{Y \conn Z}_*} \\
H_*(\CF^\infty(Y, \spincs) \otimes_{\Z[U,U^{-1}]} \CF^\infty (Z, \spinct))[2] \ar[r]^-{F^\infty_{Y \conn Z}}_-{\cong} &
  \HF^\infty (Y \conn Z)
}
\end{equation}
commute.\footnote{The notation $[2]$ means that the grading on each of the tensor products is shifted upward by $2$. This occurs because the K\"unneth formula for connected sums actually holds for $\HF^{\le 0}$, which is isomorphic to $\HF^-[2]$.} Combining this result with the algebraic K\"unneth theorem and the fact that $Y$, $Z$, and $Y \conn Z$ have standard $\HFi$, we have a commutative diagram
\begin{equation} \label{eq:connHF}
\xymatrix{
(\HF^-(Y, \spincs) \otimes_{\Z[U]} \HF^-(Z, \spinct))[2] \ar[r]^-{F^-_{Y \conn Z}} \ar[d]^{\iota_*^Y \otimes \iota_*^Z} & \HF^-(Y \conn Z, \spincs \conn \spinct) \ar[d]^{\iota_*^{Y \conn Z}}  \\
(\HF^\infty(Y, \spincs) \otimes_{\Z[U]} \HF^\infty(Z, \spinct))[2] \ar[r]^-{F^\infty_{Y \conn Z}}_-{\cong} & \HF^\infty({Y \conn Z}, {\spincs \conn \spinct}) }
\end{equation}
in which $\ker F^-_{Y \conn Z} = 0$ and $\coker F^-_{Y \conn Z}$ is a torsion $\Z[U]$--module. The images of the two vertical maps are, by definition, $(I^-(Y, \spincs) \otimes_{\Z[U]} I^-(Z, \spinct))[2]$ and $I^-(Y \conn Z, \spincs \conn \spinct)$, so we obtain a commutative diagram
\begin{equation} \label{eq:connI-}
\xymatrix{
(I^-(Y, \spincs) \otimes_{\Z[U]} I^-(Z, \spinct))[2] \ar[r]^-{F^-_{Y \conn Z}} \ar[d]^{\iota_*^Y \otimes \iota_*^Z} & I^-(Y \conn Z, \spincs \conn \spinct) \ar[d]^{\iota_*^{Y \conn Z}}  \\
(\HF^\infty(Y, \spincs) \otimes_{\Z[U]} \HF^\infty(Z, \spinct))[2] \ar[r]^-{F^\infty_{Y \conn Z}}_-{\cong} & \HF^\infty({Y \conn Z}, {\spincs \conn \spinct}) }
\end{equation}
Note that the maps $F^-_{Y \conn Z}$ and $\iota_*^Y \otimes \iota_*^Z$ in \eqref{eq:connI-} need not be injective, but they have the same kernel.

The groups in \eqref{eq:connHF} and \eqref{eq:connI-} are all modules over
\[
\Lambda^*(H_1^T(Y \conn Z)) \cong \Lambda^*(H_1^T(Y) \oplus H_1^T(Z)),
\]
and the maps are equivariant. By applying the composite functor $\QK^{V \oplus W}$ and using Lemma \ref{lemma:directsum}, we obtain a commutative diagram
\[
\xymatrix@C1in{
{\left( \begin{array}{c} \QK^V(I^-(Y, \spincs)) \\ \otimes_{\Z[U]} \\ \QK^W(I^-(Z, \spinct)) \end{array} \right)[2]} \ar[r]^-{\QK^{V \oplus W} (F^-_{Y \conn Z})} \ar[d]^{\QK^V(\iota_*^Y) \otimes \QK^W(\iota_*^Z)} &
\QK^{V\oplus W}(I^-(Y \conn Z, \spincs \conn \spinct)) \ar[d]^{\QK^{V\oplus W}(\iota_*^{Y \conn Z})}  \\
{\left( \begin{array}{c} \QK^V(\HFi(Y, \spincs)) \\ \otimes_{\Z[U]} \\ \QK^W(\HFi(Z, \spinct)) \end{array} \right)[2]} \ar[r]^-{\QK^{V \oplus W} (F^\infty_{Y \conn Z})}_-{\cong} & \QK^{V \oplus W}(\HFi({Y \conn Z}, {\spincs \conn \spinct})).}
\]

If $\QK^V(\iota_*^Y)$ is nonzero in grading $l$, and $\QK^V(\iota_*^Z)$ is nonzero in grading $m$, then $\QK^{V \oplus W}(\iota_*^{Y \conn Z})$ must be nonzero in grading $l + m +2$. Therefore,
\[
d^-(Y \conn Z, \spincs \conn \spinct, V \oplus W) \ge d^-(Y, \spincs, V) + d^-(Z, \spinct, W) + 2,
\]
so
\[
d(Y \conn Z, \spincs \conn \spinct, V \oplus W) \ge d(Y, \spincs, V) + d(Z, \spinct, W).
\]
A similar argument using $\KQ^{V \oplus W}$ in place of $\QK^{V \oplus W}$ shows that
\[
d^*(Y \conn Z, \spincs \conn \spinct, V \oplus W) \ge d^*(Y, \spincs, V) + d^*(Z, \spinct, W).
\]
Applying the same reasoning to $-(Y \conn Z) = (-Y) \conn (-W)$, we see that
\[
d(-(Y \conn Z), \spincs \conn \spinct, V \oplus W) \ge d(-Y, \spincs, V) + d(-Z, \spinct, W)
\]
and
\[
d^*(-Y \conn Z, \spincs \conn \spinct, V \oplus W) \ge d^*(-Y, \spincs, V) + d^*(-Z, \spinct, W).
\]
The desired result then follows from Proposition \ref{prop:duality}.
\end{proof}

Just as with rational homology spheres, the key property of the generalized correction terms is their behavior with respect to negative-definite cobordisms, given by a mild generalization of the results of \cite[Section 9]{oz:boundary}. To begin, note that if $X$ is a 4-manifold with possibly disconnected boundary, and $a \in H_2(X)$ is in the image of $i_*\co H_2(\partial X) \to H_2(X)$, then $a \cdot b=0$ for any $b \in H_2(X)$. Thus, the intersection pairing on $H_2(X)$ descends to a pairing on $H_2(X)/i_*(H_2(\partial X))$.

\begin{theorem} \label{thm:negdefcob}
Let $Y$ and $Y'$ be closed, oriented $3$-manifolds with standard $\HFi$, and let $W$ be an oriented cobordism from $Y$ to $Y'$ with the following properties:
\begin{enumerate}
\item The maps $H_1(Y;\Q) \to H_1(W;\Q)$ and $H_1(Y';\Q) \to H_1^T(W;\Q)$ are both surjective.
\item For any class $a \in H_2(W)$ not in the image of $H_2(\partial W) \to H_2(W)$, we have $a^2 < 0$.
\end{enumerate}
Let $V$ and $V'$ be the kernels of the maps $H_1^T(Y) \to H_1^T(W)$ and $H_1^T(Y') \to H_1^T(W)$, respectively. Let $\spincs$ be any spin$^c$ structure on $W$ whose restrictions $\spinct = \spinct|_{Y}$ and $\spinct' = \spincs|_{Y'}$ are both torsion. Then the map
\[
F^\infty_{W,\spincs}\co \HFi(Y, \spinct) \to \HFi(Y', \spinct')
\]
factors as
\[
\HFi(Y, \spinct) \twoheadrightarrow \QQ^V (\HFi(Y, \spinct)) \xrightarrow{\cong} \KK^{V'} (\HFi(Y', \spinct')) \hookrightarrow \HFi(Y', \spinct').
\]
\end{theorem}

\begin{proof}
This proceeds exactly like the proof of \cite[Theorem 9.1]{oz:boundary}.
\end{proof}

As a consequence, we see that the correction terms are invariants of the rational homology cobordism class of $Y$:

\begin{proposition} \label{prop:QHcob}
Let $Y$ and $Y'$ be closed, oriented $3$-manifolds with standard $\HFi$, and let $W$ be a rational homology cobordism from $Y$ to $Y'$ (meaning that the inclusions $i\co Y \to W$ and $i' \co Y' \to W$ induce isomorphisms on rational homology). Let $\spincs$ be any spin$^c$ structure on $W$ whose restrictions $\spinct = \spinct|_{Y}$ and $\spinct' = \spincs|_{Y'}$ are both torsion. Then for any subspaces $V \subset H_1^T(Y)$ and $V' \subset H_1^T(Y')$ such that $i_*(V \otimes \Q) = i'_*(V' \otimes \Q)$, we have
\[
d(Y, \spinct, V) = d(Y', \spinct', V') \quad \text{and} \quad d^*(Y, \spinct, V) = d^*(Y', \spinct', V').
\]
\end{proposition}

\begin{proof}
By Theorem \ref{thm:negdefcob}, the map $F_{W, \spincs} \co \HFi(Y, \spinct) \to \HFi(Y', \spinct')$ is an isomorphism. Moreover, the action of $V$ on $\HF^\circ(Y, \spinct)$ corresponds to the action of $V'$ on $\HF^\circ(Y',\spinct')$. As a result, we obtain commutative diagrams
\begin{equation} \label{eq:homcobQKV}
\xymatrix@C=0.6in{
\QK^V (\HFi(Y, \spinct)) \ar[r]^{\QK^V(\pi)} \ar[d]^{\QK^V(F^\infty_{W, \spincs_W})}_{\cong}  & \QQ (J^+(Y, \spinct, V)) \ar[d]^{\QK^V(F^+_{W, \spincs_W})} \\
\QK^{V'} (\HFi(Y', \spinct')) \ar[r]^{\QK^{V'}(\pi)} & \QQ(J^+(Y', \spinct',V'))
}
\end{equation}
and
\begin{equation} \label{eq:homcobKQV}
\xymatrix@C=0.6in{
\KQ^V (\HFi(Y, \spinct)) \ar[r]^{\KQ^V(\pi)} \ar[d]^{\KQ^V(F^\infty_{W, \spincs_W})}_{\cong}  & \KQ^V (I^+(Y, \spinct)) \ar[d]^{\KQ^V(F^+_{W, \spincs_W})} \\
\KQ^{V'} (\HFi(Y', \spinct')) \ar[r]^{\KQ^V(\pi)} & \KQ^{V'}(I^+(Y', \spinct',V')).
}
\end{equation}
Note that $c_1^2(\spincs) = \chi(W) = \sigma(W) = 0$, so each of the vertical maps shifts grading by $0$. An adaptation of the usual argument (compare \cite[proof of Theorem 9.6]{oz:boundary}) then says that
\[
d(Y, \spinct, V) \le d(Y', \spinct', V') \quad \text{and} \quad d^*(Y, \spinct, V) \le d^*(Y', \spinct', V').
\]
We may apply the same argument to $-W$, viewed as a negative semidefinite cobordism from $-Y$ to $-Y'$, to see that
\[
d(-Y, \spinct, V) \le d(-Y', \spinct', V') \quad \text{and} \quad d^*(-Y, \spinct, V) \le d^*(-Y', \spinct', V').
\]
Proposition \ref{prop:duality} then yields the desired result.
\end{proof}

\begin{corollary} \label{cor:QHS1xB3}
Let $Y$ be a closed, oriented $3$-manifold with standard $\HFi$, and let $X$ be a $4$-manifold bounded by $Y$ with $b_1(X) = b_1(Y)$ and $b_2(X)=0$. Then for any spin$^c$ structure $\spinct$ on $Y$ that extends over $X$, the correction terms of $(Y, \spinct)$ are the same as those of $\conn^{b_1(Y)} S^1 \times S^2$; i.e., for each $V \subset H_1^T(Y)$, we have
\begin{equation} \label{eq:QHS1xB3}
d(Y, \spinct, V) = \frac{n}{2} - \rank V \quad \text{and} \quad d^*(Y, \spinct, V) = -\frac{n}{2} + \rank V.
\end{equation}
\end{corollary}

\begin{proof}
Deleting a neighborhood of a bouquet of circles representing a basis for $H_1(X;\Q)$ gives a rational homology cobordism between $\conn^{b_1(Y)} S^1 \times S^2$ and $Y$; apply Proposition \ref{prop:QHcob}.
\end{proof}

More generally, the intermediate $d$ invariants of $Y$ can provide more subtle information about the intersection forms of negative-semidefinite $4$-manifolds bounded by $Y$.

\begin{theorem} \label{thm:negdef}
Let $Y$ be a closed, oriented 3-manifold with standard $\HFi$, equipped with a torsion spin$^c$ structure $\spinct$. Let $X$ be a negative-semidefinite $4$-manifold with $\partial X = Y$. Let $V_0 \subset H_1^T(Y)$ be the kernel of the map $H_1^T(Y) \to H_1^T(W)$ induced by inclusion. Then for any spin$^c$ structure $\spincs$ on $X$ whose restriction to $Y$ is $\spinct$, and any subspace $V \subset H_1^T(Y)$ that contains $V_0$, we have
\begin{equation} \label{eq:negdefdV}
c_1^2(\spincs) + b_2^-(X) \le 4d(Y, \spinct, V) - 2b_1(Y) + 4 \rank V,
\end{equation}
In particular, for $V = H_1^T(Y)$, we have
\begin{equation} \label{eq:negdefdbot}
c_1^2(\spincs) + b_2^-(X) \le 4\dbot(Y, \spinct) + 2b_1(Y).
\end{equation}
\end{theorem}

\begin{proof}
Let $n$ be the rank of the map $H_1(Y) \to H_1(X)$, so that $\rank V_0 = b_1(Y)-n$. Let $\Gamma \subset X$ be a bouquet of $n$ circles representing a basis for the image of $H_1^T(Y) \to H_1^T(W)$. Let $W$ be the $4$-manifold obtained from $X \minus \nbd(\Gamma)$ by surgering out $b_1(X)-n$ circles representing a basis for $H_1(X,Y)$. Then $W$ is a cobordism from $S_n  = \conn^n S^1 \times S^2$ to $Y$ satisfying the hypotheses of Theorem \ref{thm:negdefcob}. The kernel of $H_1^T(Y) \to H_1^T(W)$ is $V_0$. It is not hard to verify that
\[
\chi(W) = b_1(Y) - n + b_2^-(X) \quad \text{and} \quad \sigma(W) = -b_2^-(X).
\]
The spin$^c$ structure $\spincs$ induces a spin$^c$ structure on $W$, which we denote by $\spincs_W$, and $c_1^2(\spincs_W) = c_1^2(\spincs)$, and the restriction of $\spincs_W$ to $S_n$ is the unique torsion spin$^c$ structure $\spincs_0$. Thus, the maps $F_{W, \spincs_W}^\circ$ shift grading by
\[
\frac{c_1^2(\spincs) - 2b_1(Y) + 2n + b_2^-(X)}{4}.
\]

By Theorem \ref{thm:negdefcob}, the following diagram commutes:
\begin{equation} \label{eq:negdefKV0}
\xymatrix@C0.5in{
\HFi(S_n, \spincs_0) \ar[r]^{\pi} \ar[d]^{F^\infty_{W, \spincs_W}}_{\cong}  & \HFp(S_n, \spincs_0) \ar[d]^{F^+_{W, \spincs_W}} \\
\KK^{V_0} \HFi(Y, \spinct) \ar[r]^{\KK^{V_0}(\pi)} & \KK^{V_0} \HFp(Y, \spinct)
}
\end{equation}
The vertical maps respect the actions of $\Lambda^* H_1(S_n)$ and $\Lambda^*((H_1^T(Y)/V_0)/\Tors)$, which are isomorphic, on the top and bottom rows respectively. As a result, for any subspace $V \subset H_1^T(Y)$ containing $V_0$, if we let $\overline{V}$ be the subspace of $H_1^T(S_n)$ corresponding to $V$ (which is isomorphic to $V/V_0$), we obtain a commutative diagram
\begin{equation} \label{eq:negdefQKV}
\xymatrix@C0.5in{
\QK^{\overline {V}} \HFi(S_n, \spincs_0) \ar[r]^{\QK^{\overline{V}}(\pi)} \ar[d]^{F^\infty_{W, \spincs_W}}_{\cong}  & \QQ (J^+(S_n, \spincs_0, \overline {V})) \ar[d]^{F^+_{W, \spincs_W}} \\
\QK^{V} \HFi(Y, \spinct) \ar[r]^{\QK^V(\pi)} & \QQ(J^+(Y, \spinct,V)).
}
\end{equation}
just as above, whence
\[
d(Y, \spinct, V) \ge d(S_n, \spincs_0, \overline{V}) + \frac{c_1^2(\spincs) - 2b_1(Y) + 2n + b_2^-(X)}{4}.
\]
Example \ref{ex:S1xS2}, combined with the fact that $\rank (\overline{V}) = k - \rank V_0 = k - b_1(Y) +n$, implies that
\[
d(S_n, \spincs_0, \overline{V}) = \frac{n}{2} - \rank (\overline{V} ) = -\frac{n}{2} -k +b_1(Y),
\]
from which \eqref{eq:negdefdV} follows.
\end{proof}

Theorem \ref{thm:negdef} is particularly useful when $H_1(Y)$ is torsion-free, in light of Elkies' theorem characterizing the $\Z^n$ lattice \cite{elkies:lattice}.

\begin{corollary} \label{cor:intform}
Let $Y$ be a closed, oriented 3-manifold with standard $\HFi$ and with $H_1(Y) = \Z^n$, and let $\spinct_0$ be the unique torsion spin$^c$ structure on $Y$. Suppose $X$ is a negative-semidefinite $4$-manifold bounded by $Y$, with $H_1(X)$ torsion-free, and let $V = \ker(H_1(Y) \to H_1(X))$. Then
\[
d(Y, \spinct_0, V) \ge \frac{b_1(Y)}{2} - \rank(V).
\]
Moreover, if
\[
d(Y, \spinct_0, V) = \frac{b_1(Y)}{2} - \rank(V),
\]
then the intersection form on $H_2(X)/H_2(Y)$ is diagonalizable.
\end{corollary}

\begin{proof}
Consider the long exact sequence in cohomology for the pair $(X,Y)$:
\[
\cdots \to H^1(Y) \xrightarrow{\delta} H^2(X,Y) \xrightarrow{j^*} H^2(X) \xrightarrow{i^*} H^2(Y) \to \cdots
\]
As noted above, the intersection form on $H_2(X) \cong H^2(X,Y)$ descends to a form on $H_2(X)/i_*(H_2(Y)) \cong H^2(X,Y) / \delta(H^1(Y))$, which is a free abelian group since it injects into $H^2(X)$. Moreover, the short exact sequence
\[
0 \to H^2(X,Y)/\delta(H^1(Y)) \xrightarrow{j^*} H^2(X) \xrightarrow{i^*} V \to 0,
\]
splits since $H^2(Y)$, hence $V$, is torsion-free. With respect to suitable bases, the map $j^*$ is given by a matrix for the intersection form on $H^2(X)/\delta(H^1(Y))$; it follows that this intersection form is unimodular, with rank equal to $b_2^-(X)$.

Characteristic vectors for the intersection form on $H^2(X)/\delta(H^1(Y))$ are in one-to-one correspondence with spin$^c$ structures on $X$ that restrict to $\spinct_0$ on $Y$. By Elkies' theorem \cite{elkies:lattice} (see also \cite[Theorem 9.5]{oz:boundary}), we have
\[
\max \{c_1^2(\spincs) \mid \spincs \in \Spin^c(X), \spincs|_Y = \spinct_0\} \ge 0,
\]
with equality holding if and only if the intersection form is diagonalizable. The corollary then follows from Theorem \ref{thm:negdef}.
\end{proof}

As an example, suppose $X$ is a negative-semidefinite $4$-manifold bounded by the manifold $Y = S^1 \times S^2 \conn S^3_0(K)$ from Example \ref{ex:hyp}, where $K$ is the right-handed trefoil. Corollary \ref{cor:intform} says that $\ker(H_1(Y) \to H_1^T(X))$ must be either all of $H_1(Y)$ or the subgroup generated by $\beta$; in either case, the intersection form on $X$ must be diagonalizable.

\section{An application to link concordance}\label{S:links}

As noted in the introduction, the $d$ invariants for rational homology spheres have been a useful tool in the study of knot concordance, and the goal of this section is to prove analogous results for links using the generalized $d$ invariants.

Let $L=(L_1,\ldots,L_n)$ be a link in $S^3$, and denote by $\Sigma(L)$ the double cover of $S^3$ branched along $L$. We say that $L$ is \emph{(smoothly) slice} if there exist $n$ disjoint, smoothly embedded disks in $B^4$ with boundary $L$. Two links $L,L'$ are \emph{(smoothly) concordant} if there exist $n$ disjoint, smoothly embedded annuli in $S^3 \times I$ with boundary $L \times \{0\} \cup L' \times \{1\}$. Thus, $L$ is slice if and only if it is concordant to the unlink.

Denote by $\Sigma(L)$ the branched double cover of $S^3$ branched over $L$. The \emph{nullity} of $L$ is defined to be $\eta(L) = 1+ b_1(\Sigma(L))$~\cite{murasugi:numerical}. It is known~\cite{kauffman-taylor:links} that $\eta(L) \leq n$ and that $\eta$ is a concordance invariant; in particular, if $L$ is a slice link, then $\eta(L)=n$. Indeed, we have the following lemma:

\begin{lemma}[{Kauffman-Taylor~\cite[Theorem 2.6]{kauffman-taylor:links}}]
Suppose $\Delta \subset B^4$ is a union of slice disks for $L \subset S^3$. Then $\Sigma(\Delta)$ is a rational homology $\natural^n S^1 \times B^3$, and the inclusion $\Sigma(L) \to \Sigma(\Delta)$ induces an isomorphism on rational homology.
\end{lemma}

Thus, our goal is to use the tools of the previous section to obstruct $\Sigma(L)$ bounding a rational homology $\natural^n S^1 \times B^3$. To begin, we must show that $\HFi(\Sigma(L))$ is standard. By Theorem \ref{thm:Lidman}, it suffices to verify the following lemma, which dates back to Fox~\cite{fox:cyclic}. (Since Fox's paper is difficult to access, we provide the proof.)

\begin{lemma}\label{L:fox}
For any link $L$, the triple cup product on $H^1(\Sigma(L))$ vanishes identically.
\end{lemma}

\begin{proof}
Fox~\cite{fox:cyclic} showed that for a $d$-fold cyclic branched covering space, a generator $\tau$ of the group of covering transformations satisfies $1+ \tau_* + \tau^2_* + \cdots \tau^{d-1}_* =0$ on $H_1$; this is a nice exercise using properties of the transfer sequence.  In particular for $2$-fold covers, $\tau_*$ acts by $-1$.  It follows that for any $\alpha, \beta, \gamma \in H^1(\Sigma(L))$, we have
\begin{align*}
\langle \alpha \cup \beta \cup \gamma, [\Sigma(L)] \rangle & = \langle \alpha \cup \beta \cup \gamma, \tau_*[\Sigma(L)] \rangle\\
&=\tau^*(\langle \alpha \cup \beta \cup \gamma), [\Sigma(L)] \rangle\\
&= -\langle \alpha \cup \beta \cup \gamma, [\Sigma(L)] \rangle
\end{align*}
and so the triple product must vanish.
\end{proof}



According to Turaev~\cite{turaev:montesinos}, the choice of an orientation $o$ on $L$ determines a canonical spin structure $\spincs_o$ on $\Sigma(L)$, so that an orientation gives rise to a parametrization of \spinc\ structures on $\Sigma(L)$ as $\spincs_c = \spincs_o + c$ for $c \in H^2(\Sigma(L))$. In particular, $\spincs_c$ is torsion if and only if $c$ is torsion. By~\cite{donald-owens:links}, for any oriented surface $F \subset B^4$ (connected or not) with boundary $L$, the spin structure $\spincs_o$ extends over $\Sigma(F)$, and therefore $\spincs_c$ extends over $W$ if and only if $c \in \im(j^*\co H^2(\Sigma(F)) \to H^2(\Sigma(L)))$, or equivalently if $\PD(c) \in \ker(j_*\co H_1(\Sigma(L)) \to H_1(\Sigma(F)))$. Thus, we must understand the torsion elements in this kernel.

\begin{lemma} \label{lemma:metabolizer}
Suppose that $Y$ is the boundary of $W$, rational homology $\natural^n S^1 \times B^3$. Let $A$ denote the kernel of the map $T_1(Y) \to H_1(W)$, where $T_1(Y)$ denotes the torsion subgroup of $H_1(Y)$. Then $A$ is a metabolizer for the linking form
\[
T_1(Y) \otimes T_1(Y) \to \Q/\Z;
\]
in particular, we have $\abs{A}^2 = \abs{T_1(Y)}$.
\end{lemma}

\begin{proof}
The lemma is well-known~\cite[Lemma 4.3]{casson-gordon:stanford} in the case when $Y$ is a rational homology sphere and $W$ is a rational homology ball (i.e. $n=0$). In the general case, choose embedded curves representing a basis for $H_1^T(Y)$, and add handles to $W$ along those curves to get a rational homology ball, denoted $W'$. Its boundary is $Y'$, a rational homology sphere.  It is easy to check that $T_1(Y) = T_1(Y')$ and $T_1(W) = T_1(W')$ in such a way that $A$ is isomorphic to $\ker j'_*: T_1(Y') \to H_1(W')$. Moreover, the linking forms of $Y$ and $Y'$ are isomorphic.  By the (usual) result for $n=0$, $A$ is a metabolizer as claimed.
\end{proof}

Assembling these results, we obtain the main theorem of this section, which was well known in the case of knots \cite{jabuka-naik:dcover}.

\begin{theorem}\label{T:link-d}
Suppose $L \subset S^3$ is smoothly slice. Then there is a subgroup $A \subset T_1(\Sigma(L))$ that is a metabolizer for the linking form
\[
\lambda:  T_1(\Sigma(L)) \otimes  T_1(\Sigma(L)) \to \Q/\Z
\]
such that for any \spinc\ structure $\spincs_t = \spincs_o + t$ with $t \in A$, we have $\dbot(\Sigma(L), \spincs_t) = -n/2$ and $\dtop(\Sigma(L), \spincs_t) = n/2$, and hence
\[
d(\Sigma(L),\spincs,V) = \frac{n}{2} - \rank(V) \quad \text{and} \quad d^*(\Sigma(L),\spincs,V) = \frac{n}{2} + \rank(V)
\]
for any subspace $V \subset H_1^T(\Sigma(L))$.
\end{theorem}

\begin{proof}
Let $\Delta$ be a union of slice disks for $L$. The metabolizer $A$ is provided by Lemma \ref{lemma:metabolizer}. For any $t\in A$, the $\spinc$ structure $\spincs_o +t$ extends over $\Sigma(\Delta)$, so the result follows by Corollary~\ref{cor:QHS1xB3}.
\end{proof}

At present, all examples of links that we can show to be non-slice using Theorem~\ref{T:link-d} can be be shown to be non-slice by some other, perhaps simpler, method. There are many links whose smooth concordance class is unknown (such as the topologically slice examples of Cochran, Friedl, and Teichner~\cite{cochran-friedl-teichner:newslice}), but the calculation of the requisite $d$-invariants for these examples seems challenging.

\renewcommand{\MR}[1]{}

\end{document}